\documentclass[12pt]{amsart}

\date{3 May 2016}
 \title{A finite version of the Kakeya problem} 
 \author[Ball]{Simeon Ball}\thanks{The first author acknowledges the support of the project MTM2014-54745-P of the Spanish {\em Ministerio de Economia y Competitividad.}} 
  \author[Blokhuis]{Aart Blokhuis}
   \author[Domenzain]{Diego Domenzain}
 \usepackage{amsmath}
\usepackage{amssymb}
\usepackage{amsthm}
\usepackage{graphicx}
\usepackage{amscd}

\newtheorem{theorem}{Theorem}[section] 
\newtheorem{lemma}[theorem]{Lemma}     

\newtheorem{example}{Example}

 \makeatletter
 \@addtoreset{equation}{section}
 \makeatother

\begin{document}
 
 \begin{abstract}
Let $L$ be a set of lines of an affine space over a field and let $S$ be a set of points with the property that every line of $L$ is incident with at least $N$ points of $S$. Let $D$ be the set of directions of the lines of $L$ considered as points of the projective space at infinity. We give a geometric construction of a set of lines $L$, where $D$ contains an $N^{n-1}$ grid and where $S$ has size $2(\frac{1}{2}N)^n$ plus smaller order terms, given a starting configuration in the plane. We provide examples of such starting configurations for the reals and for finite fields. Following Dvir's proof of the finite field Kakeya conjecture and the idea of using multiplicities of Dvir, Kopparty, Saraf and Sudan, we prove a lower bound on the size of $S$ dependent on the ideal generated by the homogeneous polynomials vanishing on $D$. This bound is maximised as $(\frac{1}{2}N)^n$ plus smaller order terms, for $n\geqslant 4$, when $D$ contains the points of a $N^{n-1}$ grid.
 \end{abstract}
 
\maketitle

 
\section{Introduction}
 
Let $\mathrm{AG}_n({\mathbb K})$ denote the $n$-dimensional affine space over the field ${\mathbb K}$ and let $\mathrm{PG}_n({\mathbb K})$ denote the $n$-dimensional projective space over the field ${\mathbb K}$. 

Let $L$ be a set of lines of $\mathrm{AG}_n({\mathbb K})$ and let $D$ be the set of directions of the lines of $L$, viewed as points of the projective space $\mathrm{PG}_{n-1}({\mathbb K})$ at infinity. Let $S$ be a set of points of $\mathrm{AG}_n({\mathbb K})$ with the property that every line of $L$ is incident with at least $N$ points of $S$. 

In the case that ${\mathbb K}={\mathbb F}_q$ and $N=q$, Dvir \cite{Dvir2009} proved that if $D$ is the set of all directions then $|S| > q^n/n!$, answering a question posed by Wolff in \cite{Wolff1999}. Dvir, and subsequently Saraf and Sudan \cite{SS2008}, provided examples where $D$ is the set of all directions and sets of points $S$ for which $|S|=2(\frac{1}{2}q)^n$ plus smaller order terms. The lower bound on $|S|$ was improved for $n\geqslant 4$ to $|S| \geqslant (\frac{1}{2}q)^n+c(n)q^{n-1}$ for some $c(n)$, by Dvir, Kopparty, Saraf and Sudan in \cite{DKSS2013}.

We define a $N^{n-1}$ {\em grid} in $\mathrm{PG}_{n-1}({\mathbb K})$ as a point set, which with respect to a suitable basis, has the form
$$
\{ \langle (a_1,\ldots,a_{n-1},1) \rangle \ | \ a_i \in A_i \},
$$
where $A_i$ is a subset of ${\mathbb K}$ of size $N$ for all $i=1,\ldots,n-1$. Here, we introduce the notation that if $(a_1,\ldots,a_n)$ is a vector, then $\langle (a_1,\ldots,a_n) \rangle$ is the corresponding point in $\mathrm{PG}_{n-1}({\mathbb K})$.

The aim of this article is to reformulate the Kakeya problem in a far more general setting. There are a couple of recent articles by Slavov in which he formulates the Kakeya problem in an algebraic geometric setting, see \cite{Slavov2014a} and \cite{Slavov2014b}. Here we consider any arbitrary finite set of lines in an affine space over any fixed field ${\mathbb K}$. We prove lower bounds on the size of $S$ that depend on $I(D)$, the ideal generated by the homogeneous polynomials of ${\mathbb K}[X_1,\ldots,X_n]$ which are zero at all points of $D$. Firstly, we will give a geometric construction of a set $L$ of $N^{n-1}$ lines, whose directions contain a $N^{n-1}$ grid and a set $S$ of roughly $2(\frac{1}{2}N)^{n}$ points with the property that every line of $L$ is incident with at least $N$ points of $S$.

\section{A geometric construction of Kakeya sets}

For any two non-intersecting subspaces $x$ and $y$ of a projective space we denote by $x\oplus y$ the subspace that they span.

Let $x_0,x_1,\ldots,x_n$ be projective points in general position which will remain fixed throughout. Let 
$$
\Sigma_i=x_0 \oplus x_1 \oplus \cdots \oplus x_i,
$$
and let
$$
\pi_i=x_1 \oplus x_2 \oplus \cdots \oplus x_i,
$$
for $i=1,\ldots,n$. Then $\Sigma_n$ is the entire space $\mathrm{PG}_n({\mathbb K})$, $\pi_n$ we consider as the hyperplane at infinity and $\Sigma_n \setminus \pi_n$ is the affine space $\mathrm{AG}_n({\mathbb K})$ where we shall construct the set of lines $L'$.

Let $y_i$ be a third point on the line $x_{i-1} \oplus x_i$, for $i=3,\ldots,n$, so
$$
(x_i\oplus y_i) \cap \pi_{i-1}=x_{i-1}.
$$

Let $L$ be a set of lines of $\Sigma_2$, incident with distinct points of $\pi_2 \setminus \{x_2\}$. So by interpreting $\pi_2$ as the line at infinity, all the lines in $L$ have distinct directions. Note that if ${\mathbb K}$ is infinite then we can always find a line which intersects the lines of a finite set of lines in distinct points. More generally for a set of lines in $\mathrm{AG}_n({\mathbb K})$, with ${\mathbb K}$ infinite, we can always find a hyperplane which intersects the lines of a finite set of lines in distinct points. After a suitable change of basis this hyperplane can be assumed to be the hyperplane at infinity. 

Label the lines of $L$ so that they are $\ell_{\{1\}},\ldots,\ell_{\{|L|\}}$, and define 
$$
p_{\{i\}}=\ell_{\{i\}} \cap \pi_2.
$$
By assumption, $p_{\{1\}},\ldots,p_{\{|L|\}}$ are distinct points of $\pi_2$.



For an ordered subset $J$ of $\{1,\ldots,|L|\}$, $|J|<n$, we define a line $\ell_J$ of $\Sigma_{|J|+1}$ recursively by
$$
\ell_{J}=(x_{|J|+1} \oplus \ell_{J \setminus \{a \}}) \cap (y_{|J|+1} \oplus \ell_{J \setminus \{b \}}), 
$$
where $a$ and $b$ are the last two elements of $J$, so $J=(\ldots,b,a)$.

In the same way, we define points $p_J$ of $\pi_{|J|+1}$ recursively by
$$
p_{J}=(x_{|J|+1} \oplus p_{J \setminus \{a \}}) \cap (y_{|J|+1} \oplus  p_{J \setminus \{b \}}).
$$

\begin{lemma} \label{welldef}
The lines $\ell_J$ are well-defined and distinct, as are the points $p_J$. Furthermore, the line $\ell_J$ intersects $\pi_{|J|+1}$ in the point $p_J$.
\end{lemma}

\begin{proof}
By induction on $|J|$. For $|J|=2$, let $J=\{b,a\}$. The line $\ell_J$ is well-defined, since $x_3 \oplus \ell_{\{b\}}$ and $y_3 \oplus \ell_{\{a\}}$ are two distinct planes in a $3$-space and therefore intersect in a line. For $|J| \geqslant 3$, suppose $J=\{\ldots,b,a\}$. The lines $\ell_{J \setminus \{a \}}$ and $\ell_{J \setminus \{b \}}$ are both contained in the plane $\ell_{J \setminus \{a,b \}} \oplus x_{|J|}$, so intersect in a point. Moreover $x_{|J|+1} \oplus y_{|J|+1}$ contains the point $x_{|J|}$ and $x_{|J|}$ is in the plane $\ell_{J \setminus \{a \}} \oplus \ell_{J \setminus \{b \}}$, since $\ell_{J \setminus \{a \}} \oplus \ell_{J \setminus \{b \}}=\ell_{J \setminus \{a,b \}} \oplus x_{|J|}$. Therefore, $x_{|J|+1} \oplus \ell_{J \setminus \{a \}}$ and $y_{|J|+1} \oplus \ell_{J \setminus \{b \}}$ are distinct planes contained in a $3$-space. Hence, their intersection is a line. 

By induction, the point $p_J$ is distinct for distinct $J$ since it is the intersection of the line $x_{|J|+1} \oplus p_{J \setminus \{a \}}$ and 
$y_{|J|+1} \oplus p_{J \setminus \{b \}}$ and $J \setminus \{a \}$ and $J \setminus \{b \}$ determine $J$ (knowing they are both ordered $(|J|-1)$-subsets of $J$ determines $J$).
Since $\ell_J$ is incident with the point $p_J$ it follows that $\ell_J$ is also distinct for distinct $J$.
\end{proof}

\begin{lemma} \label{jgrid}
The set of points 
$$
\{ p_J \ | \ J \ \mathrm{ordered} \ \mathrm{subset} \ \mathrm{ of} \ L,\ |J|=n \}
$$ 
is contained in a $|L|^{n}$ grid of $\pi_{n+1}$.
\end{lemma}

\begin{proof}
Fix a basis of $\pi_{n+2}$ so that $x_1,\ldots,x_{n+2}$ are points derived from the canonical basis. For $i=3,\ldots,n+2$, let $y_i=\langle (0,\ldots,0,1,1,0,\ldots,0)\rangle$, where the non-zero coordinates are the $(i-1)$-th and $i$-th coordinate.

Let $J=\{a_1,\ldots,a_n \}$ and define $d_i$ by
$$
\ell_{\{a_i\}} \cap \pi_2 = \langle (1,d_i,0\ldots,0)\rangle,
$$
for $i=1,\ldots,j$. Note that $\{ d_1,\ldots,d_j \}$ is a subset of
$$
D=\{ d \ | \ \mathrm{there} \ \mathrm{exists} \ \ell \in L \ \mathrm{such} \ \mathrm{that} \ \ell \cap \pi_2=\langle (1,d,0,\ldots,0) \rangle \}.
$$

We will prove by induction that 
$$
p_J=\langle (1,d_1,\ldots) \rangle
$$
and that the $i$-th coordinate of $p_J$ for $|J|+1 \geqslant i\geqslant 3$ is $(-1)^i(d_{i-1}-d_{i-2})$ and zero for $i\geqslant |J|+2$.

Let $J=\{ a_1,\ldots,a_j\}$, $\overline{J}=\{a_1,\ldots,a_j,a_{j+1} \}$ and $J'=\{a_1,\ldots,a_{j-1},a_{j+1} \}$. Then by definition
$$
p_{\overline{J}}=(p_J \oplus x_{j+2}) \cap (p_{J'} \oplus y_{j+2}).
$$
Since the first $j$ coordinates of $x_{j+2}$ and $y_{j+2}$ are zero, the $i$-th coordinate of $p_{\overline{J}}$ for $i\leqslant j$ is $(-1)^i(d_{i-1}-d_{i-2})$. Since the $(j+1)$-st coordinate of $x_{j+2}$ is zero, the $(j+1)$-st coordinate of $p_{\overline{J}}$ is $(-1)^{j+1}(d_{j}-d_{j-1})$. If $\lambda$ is the $(j+2)$-nd coordinate then solving the equation given by the intersection (comparing the $(j+1)$-st coordinate) gives
$$
(-1)^{j+1}(d_{j}-d_{j-1})=\lambda+(-1)^{j+1}(d_{j+1}-d_{j-1}),
$$
which proves the induction.

By applying a simple change of basis we see that the set 
$$
\{ p_J \ | \ J \ \mathrm{ordered} \ \mathrm{subset} \ \mathrm{ of} \ L,\ |J|=j \}
$$ 
is contained in the $|L|^{j}$ grid,
$$
\{\langle (1,e_1,\ldots,e_j) \rangle \ | \ e_1,\ldots,e_j \in D \}.
$$
\end{proof}

\begin{theorem} \label{setoflines}
The set of lines
$$
\{ \ell_J \ | \ J \ \mathrm{ordered} \ \mathrm{subset} \ \mathrm{ of} \ L,\ |J|=j \}
$$ 
is a set of lines in $\Sigma_{j+1}$ whose directions are distinct and contained in a $|L|^{j-1}$ grid of $\pi_{j+1}$.
\end{theorem}

\begin{proof}
This follows from the definitions of $\ell_J$ and $p_J$, Lemma~\ref{welldef} and Lemma~\ref{jgrid}.
\end{proof}

Let $m$ be a line of $\Sigma_2$ incident with the point $x_2$ and not the line $\pi_2$. If the line $\ell_{\{i\}}$ and $\ell_{\{j\}}$ meet on the line $m$ then define their intersection to be the point
$$
z_{\{i\},\{j\},m}=\ell_{\{i\}} \cap \ell_{\{j\}} \cap m.
$$

Let $J=\{a_1,\ldots,a_j\}$ and $\overline{J}=\{\overline{a}_1,\ldots,\overline{a}_j\}$ be disjoint ordered subsets of $\{1,\ldots,|L|\}$ such that
$$
\ell_{\{a_i\}} \cap \ell_{\{\overline{a}_i\}} \cap m
$$
is a point for all $i=1,\ldots,j$. For each such occurrence of $m$, $J$ and $\overline{J}$ define a point recursively by
$$
z_{J,\overline{J},m}=(x_{|J|+1} \oplus z_{J \setminus \{a\},\overline{J}\setminus \{\overline{a}\},m}) \cap (y_{|J|+1} \oplus z_{J \setminus \{b\},\overline{J}\setminus \{\overline{b}\},m})
$$
where $J=(\ldots,b,a)$ and $\overline{J}=(\ldots,\overline{b},\overline{a})$.

\begin{lemma} \label{zonl}
The point $z_{J,\overline{J},m}$ is incident with the line $\ell_J$.
\end{lemma}

\begin{proof}
By induction on $|J|$. For $|J|=1$, this follows directly from the definition. 

By the induction hypothesis, we suppose $z_{J \setminus \{a\},\overline{J}\setminus \{\overline{a}\},m}$ is incident with $\ell_{J \setminus \{a\}}$ and $z_{J \setminus \{b\},\overline{J}\setminus \{\overline{b}\},m}$ is incident with $\ell_{J \setminus \{b\}}$. Then, by the definition of $\ell_J$, the point $z_{J,\overline{J},m}$ is incident with the line $\ell_J$.
\end{proof}

\begin{lemma} \label{expo}
Let $J$ and $\overline{J}$ be disjoint ordered non-empty subsets of $\{1,\ldots,|L|\}$ of the same size. Suppose $J'$ and $\overline{J}'$ are disjoint ordered $|J|$-subsets of $\{1,\ldots,|L|\}$ where the $i$-th element of $J'$ and $\overline{J}'$ is either the $i$-th element of $J$ or the $i$-th element of $\overline{J}$, for $i=1,\ldots,|J|$. Then
$$
z_{J,\overline{J},m}=z_{J',\overline{J}',m}.
$$
\end{lemma}

\begin{proof}
For $|J|=1$ this is clear. It follows from the recursive defintion of $z_{J,\overline{J},m}$ that switching the $i$-th element of $J$ and $\overline{J}$ will not affect the point $z_{J,\overline{J},m}$.
\end{proof}

\begin{figure}[h]
\centering
\includegraphics[width=5.6 in]{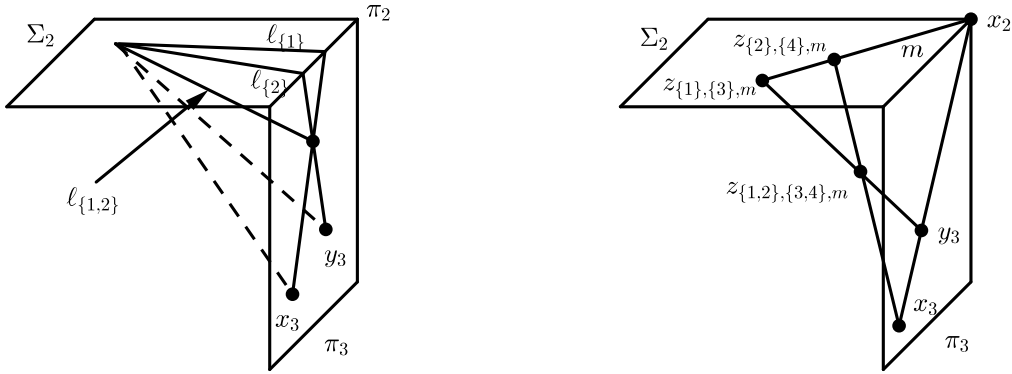}
\caption{The construction of the line $\ell_{\{1,2\}}$ and the point $z_{\{1,2 \},\{3,4\},m}$}.
\label{geomkak}
\end{figure}

\begin{theorem} \label{geomcon}
Suppose that $L$ is a set of $N$ lines of $\mathrm{AG}_2({\mathbb K})$ and let $S$ be a set of points with the property that every line of $L$ is incident with $N$ points of $S$. Suppose that there are $N$ parallel lines $m_i$, which are incident with $\frac{1}{2}N-\epsilon_i$ points of $S$ which themselves are incident with two lines of $L$, where $\epsilon_1,\ldots,\epsilon_N$ have the property that
$$
\sum_{i=1}^N \epsilon_i \leqslant dN,
$$ 
for some constant $d$, not depending on $N$. 

Then there is a set $L'$ of $N^{n-1}$ lines in $\mathrm{AG}_n({\mathbb K})$, $n \leqslant \frac{1}{2}N+1$, whose directions contain a $N^{n-1}$ grid and a set of points $S'$ with the property that every line of $L'$ is incident with $N$ points of $S'$ and where $S'$ has less than $2(\frac{1}{2}N)^n+cN^{n-1}$ points, for some $c=c(n)$. 
\end{theorem}

\begin{proof}
Let 
$$
L'=\{ \ell_M \ | \ M \ \mathrm{ordered} \ \mathrm{subset} \ \mathrm{ of} \ L,\ |M|=n-1 \}
$$ 
and let
$$
S'=\{ z_{M,\overline{M},m} \ | \ M,\overline{M} \ \mathrm{ordered} \ \mathrm{subsets} \ \mathrm{ of} \ L,\ |M|=|\overline{M}|=n-1,\ m=m_i,\ \mathrm{for} \ \mathrm{some} \ i \}.
$$
The point $ z_{M,\overline{M},m}$ is only defined if $M$ and $\overline{M}$ are disjoint, which imposes the condition $N \geqslant 2(n-1)$.

By Theorem~\ref{setoflines}, the set $L'$ contains $N(N-1)\ldots(N-n+2)$ lines and the lines determine distinct directions contained in a $N^{n-1}$ grid.

The set $S'$ contains 
$$
\sum_{i=1}^N(\tfrac{1}{2}N-\epsilon_i)(\tfrac{1}{2}N-\epsilon_i-1)\cdots (\tfrac{1}{2}N-\epsilon_i-n+2)
$$
points. By Lemma~\ref{zonl} and Lemma~\ref{expo}, a point of $S'$ is incident with $2^{n-1}$ lines of $L'$. Therefore, we have constructed at least $N^n-c(n)N^{n-1}$ incidences between lines of $L'$ and points of $S'$, for some $c(n)$. 

The set $L'$ contains $N(N-1)\ldots(N-n+2)$ lines and we would like each line to be incident with $N$ points of $S'$. Therefore we are missing less than $c(n)N^{n-1}$ incidences. For a fixed $M$ and $m$ the point $z_{M,\overline{M},m}$ does not depend on $\overline{M}$, since it is defined recursively from the points $\ell_{\{ i\}} \cap m$, where $i \in M$. Therefore, no line of $L'$ is incident with more than $N$ points of $S'$ (at most one for each line $m_i$, $i=1,\ldots,N$). So we can add less than $c(n)N^{n-1}$ points to $S'$ so that every line of $L'$ is incident with $N$ points of $S'$. This does not affect the first order term of $|S'|$, which is $N^n/2^{n-1}$. 

Finally, we add lines to $L'$ and $N$ points to $S'$ for each of these lines, so that we have a line with every direction of the $N^{n-1}$ grid. Thus far we have constructed ${N \choose n-1}(n-1)!$ lines in $L'$, so we add less than $c'N^{n-2}$ lines to $L'$ to complete the grid and add at most $c'N^{n-1}$ points to $S'$, for some $c'=c'(n)$. Again, this does not affect the first order term of $|S'|$.
\end{proof}

\begin{example}
\rm{If ${\mathbb K}={\mathbb F}_q$ and $N=q$ then we can take $L$ to be the lines of a dual conic (or any oval), where one of the lines is taken to be the line at infinity $\pi_2$. The points of $S$ will include the affine points incident with a line of $L$.

Let $x$ be the point incident with $\pi_2$ and not incident with a line of $L$. The lines $m_1,\ldots,m_N$ will be the $q$ affine lines incident with $x$. Suppose $q$ is odd. Since each point not on the conic but incident with a tangent to the conic is incident with $(q-1)/2$ bisecants, we have $\epsilon=\frac{1}{2}$ for all $i=1,\ldots,q$ before we add points to $S$. Adding $N$ points to $S$ does not affect the fact that the condition on the $\epsilon_i$. If $q$ is even then each point not on the conic but incident with a tangent to the conic is incident with $q/2$ bisecants, except one point which is incident with no bisecants. Therefore, $\epsilon_i=0$ for $i=1,\ldots,q-1$ and $\epsilon_q=\frac{1}{2}q$. Again, adding $N$ points to $S$ does not affect the condition on the $\epsilon_i$.}
\end{example}

\begin{example} \label{reals}
\rm{If ${\mathbb K}={\mathbb R}$ then we can take $L$ to be the set of lines dual to a regular $N$-gon. We dualise in such a way that the line at infinity becomes a point on the line at infinity. Let $S$ be the set of affine points dual to the bisecants to the $N$-gon. This gives $N-1$ points on each line of $L$ and we arbitrarily add an additional point to $S$ incident with $\ell$, for each line $\ell \in L$. 

The line joining $(\cos(2\pi a/N),\sin(2\pi a/N),1)$ and $(\cos(2\pi b/N),\sin(2\pi b/N),1)$ meets the line at infinity in the point $(-\tan(\pi (a+b)/N),1,0)$ (\cite[Proposition 2.1]{GT2013}), so there are precisely $N$ points on the line at infinity where the bisecants meet. 

Let $p_1,\ldots,p_N$ be the $N$ points on the line at infinity where the bisecants meet. Let $m_1,\ldots,m_N$ be the $N$ (parallel) lines dual to the points $p_1,\ldots,p_N$. Before we add points to $S$ we have that if $N$ is even then $\epsilon_i=0$ for $i=1,\ldots,\frac{1}{2}N$ and $\epsilon_i=1$ for $i=\frac{1}{2}N+1,\ldots,N$ and if $N$ is odd then $\epsilon_i=\frac{1}{2}$ for $i=1,\ldots,N$, ordering the lines in a suitable way. Adding $N$ points to $S$ does not affect the condition on $\epsilon_1,\ldots,\epsilon_N$.}
\end{example}

In \cite{GK2010}, Guth and Katz prove that if $L$ is a set of lines in $\mathrm{AG}_3({\mathbb R})$, no $N$ of which are contained in a plane, and if $S$ is a set of points with the property that every line of $L$ is incident with at least $N$ points of $S$, then $|S|>cN^3$ for some (very small) constant $c$. Example~\ref{reals}, together with Theorem~\ref{geomcon}, provide an example of such a set of lines for which $|S|=\frac{1}{4}N^3$ plus smaller order terms. 

\section{A lower bound for $|S|$}

The proofs in this section are essentially from \cite{DKSS2013}. Although in \cite{DKSS2013} they restrict to the case that the field is finite, in this section we verify that the proofs carry over to the general case without any issue.

Let $\mathcal{J}=({\mathbb Z}_{\geqslant 0})^n$ be the set $n$-tuples of non-negative integers. For any $j,c \in \mathcal{J}$, we define the $j$-Hasse derivative of $X^c=\prod_{i=1}^n X_i^{c_i}$ as
$$
\partial^j(X^c)=\prod_{i=1}^n {c_i \choose j_i} X_i^{c_i-j_i},
$$
where we use the convention ${a \choose b}=0$ if $b>a$. This definition extends to polynomials by linearity. 

For a polynomial $f \in {\mathbb K}[X_1,\ldots,X_n]$, we define $V(f)$ to be the affine points which are zeros of $f$.

For any $j \in \mathcal{J}$, let $\mathrm{wt}(j)=\sum_{i=1}^n j_i$. We say that a polynomial $f \in {\mathbb K}[X_1,\ldots,X_n]$ has a {\em zero of multiplicity $m$} at a point $u$ of $\mathrm{AG}_n({\mathbb K})$ if $u \in V(\partial^j f)$ for all $j \in \mathcal{J}$, where $\mathrm{wt}(j) \leqslant m-1$. 

If $u$ is a zero of multiplicity $m$ of $f$ and $\mathrm{wt}(j)=r$ then $u$ is a zero of multiplicity at least $m-r$ of $\partial^j f$, see \cite{DKSS2013}.

Let $L$ be a set of lines of $\mathrm{AG}_n({\mathbb K})$ and let $D$ be the set of directions of the lines of $L$, viewed as points of the projective space $\mathrm{PG}_{n-1}({\mathbb K})$ at infinity. Let $S$ be a set of points of $\mathrm{AG}_n({\mathbb K})$ with the property that every line of $L$ is incident with at least $N$ points of $S$. 

Let $I_r(D)$ be the ideal of homogeneous polynomials of ${\mathbb K}[X_1,\ldots,X_n]$ which have zeros of multiplicity at least $r$ at all points of $D$. 

For any $f \in {\mathbb K}[X_1,\ldots,X_n]$, let $f^*$ denote the polynomial consisting of the terms of $f$ of highest degree. By degree, we will always mean the total degree.

\begin{theorem} \label{zeevetal}
If $U$ is a subspace of ${\mathbb K}[X_1,\ldots,X_n]$ of polynomials of degree at most $rN-1$ with the property that for all non-zero $f \in U$, $f^* \not\in I_{r}(D)$ then
$$
{2r+n-2 \choose n}|S|\geqslant \dim U.
$$
\end{theorem}

\begin{proof}
Suppose that ${2r+n-2 \choose n}|S| < k$, where $\dim U=k$. Let $f(X)$ be a polynomial of $U$, so
$$
f(X)=\sum_{i=1}^k a_ih_i(X),
$$
where $\{h_1,\ldots,h_k\}$ is a basis for $U$, for some $a_i \in {\mathbb K}$. 

We wish to show that there is a non-zero polynomial in $U$ which has a zero of multiplicity at least $2r-1$ at all points $x \in S$. A polynomial $f$ has a zero of multiplicity at least $2r-1$ at all points $x \in S$ if and only if for every $x \in S$, $\partial^j f$ has a zero of multiplicity $2r-1-\mathrm{wt}(j)$ at $x$. For each $j \in \mathcal{J}$, where $\mathrm{wt}(j) \leqslant 2r-2$, $\partial^j f(x)=0$ is a linear homogeneous equation with unknowns $a_1,\ldots,a_k$. Thus we get a system of ${n+2r-2 \choose n}|S|$ linear homogeneous equations and $k$ unknowns. Since ${n+2r-2 \choose n}|S|<k$ there must be a non-trivial solution and so there is a non-zero $f \in U$ such that $f$ has a zero of multiplicity at least $2r-1$ at all points $x \in S$.

Let $v$ be a vector of the $n$-dimensional vector space, such that the subspace spanned by $v$ is an element of $D$. By hypothesis, there is a $u \in \mathrm{AG}_n({\mathbb K})$ and $N$ distinct values $\lambda \in {\mathbb K}$ with the property that $u+\lambda v$ is a zero of $f$ of multiplicity at least $2r-1$. For any $j \in \mathcal{J}$ with $\mathrm{wt}(j) \leqslant r-1$, $u+\lambda v$ is a zero of $\partial^j f$ of multiplicity at least $r$. Since $d=\deg f \leqslant rN-1$, it follows that $\partial^j f(u+\lambda v)$ is identically zero as a polynomial in $\lambda$. The coefficient of $\lambda^d$ of $\partial^j f$ is $\partial^j f^*(v)$. Hence, $\partial^j f^*(v)=0$ for all $j \in \mathcal{J}$, where $\mathrm{wt}(j) \leqslant r-1$. This implies $f^*$ has a zero of multiplicity $r$ at all points of $D$ and so $f^* \in I_r(D)$,  which is a contradiction, since $f \in U$ implies $f^*\not\in I_r(D)$.
\end{proof}

Theorem~\ref{zeevetal} allows us to give an explicit lower bound for $|S|$ if $D$ contains an $N^{n-1}$ grid.

Recall that we defined a $N^{n-1}$ grid in $\mathrm{PG}_{n-1}({\mathbb K})$ as
$$
\{ \langle (a_1,\ldots,a_{n-1},1) \rangle \ | \ a_i \in A_i \},
$$
where $A_i$ is a subset of ${\mathbb K}$ of size $N$ for all $i=1,\ldots,n-1$.

\begin{theorem} \label{grid}
If $D$ contains an $N^{n-1}$ grid then, for any $r\in {\mathbb N}$,
$$
{2r+n-2 \choose n}|S| \geqslant {rN+n-1 \choose n}.
$$
\end{theorem}

\begin{proof}
The ideal $I_r(D)$ is generated by products of $r$ (not necessarily distinct) polynomials from the set $\{ g_1,\ldots,g_{n-1}\}$, where
$$
g_i(X)=\prod_{a \in A_i} (X_i-aX_n). 
$$
Any non-zero polynomial in $I_r(D)$ has degree at least $rN$, so we can set $U$ to be the subspace of all polynomials in ${\mathbb K}[X_1,\ldots,X_n]$ of degree at most $rN-1$. 
\end{proof}

\begin{theorem} \label{grid2}
If $D$ contains an $N^{n-1}$ grid then 
$$
|S| \geqslant {N+n-1 \choose n}.
$$
\end{theorem}

\begin{proof}
Put $r=1$ in Theorem~\ref{grid}.
\end{proof}

The following theorem improves on the lower bound in Theorem~\ref{grid2} for $n\geqslant 4$.

\begin{theorem} \label{grid3}
If $D$ contains an $N^{n-1}$ grid then $|S| \geqslant (\frac{1}{2}N)^n$.
\end{theorem}

\begin{proof}
By Theorem~\ref{grid} we have
$$
|S| \geqslant \frac{(rN+n-1)(rN+n-2)\ldots (rN)}{(2r+n-2)(2r+n-3)\ldots (2r-1)}.
$$
This gives $|S| \geqslant (\frac{1}{2}N)^n$ if we choose $r$ large enough.

\end{proof}

\section{Acknowledment}
We would like to thank the anonymous referees who made various suggestions and helpful comments, vastly improving the readability of the article.

\bigskip

{\small Simeon Ball}  \\
{\small Departament de Matem\`atiques}, \\
{\small Universitat Polit\`ecnica de Catalunya, Jordi Girona 1-3},
{\small M\`odul C3, Campus Nord,}\\
{\small 08034 Barcelona, Spain} \\
{\small {\tt simeon@ma4.upc.edu}}

{\small Aart Blokhuis} \\
{\small Department of Mathematics and Computing Science}, \\
{\small Eindhoven University of Technology},
{\small P.O. Box 513 \\
{\small 5600MB Eindhoven, The Netherlands}\\
{\small {\tt aartb@win.tue.nl}}

{\small Diego Domenzain} \\
{\small Michigan Technological University,}\\
{\small 630 Dow Environmental Sciences And Engineering Building,}\\
{\small 1400 Townsend Drive, Houghton,}\\
{\small Michigan 49931, United States}\\


\begin{thebibliography}{}

\bibitem{Dvir2009} Z. Dvir, On the size of {K}akeya sets in finite fields, {\em J. Amer. Math. Soc.}, {\bf 22} (2009) 1093--1097.
     
\bibitem{DKSS2013} Z. Dvir, Swastik Kopparty, Shubhangi Saraf and Madhu Sudan, Extensions to the method of multiplicities, with applications to {K}akeya sets and mergers, {\em SIAM J. Comput.}, {\bf 42} (2013) 2305--2328.

\bibitem{GT2013} B. Green and T. Tao, On sets defining few ordinary lines, {\it Discrete and
Computational Geometry}, {\bf  50} (2013) 409--468.

\bibitem{GK2010}
L. Guth and N. Katz, 
Algebraic methods in discrete analogs of the {K}akeya problem,
{\it Adv. Math.},
{\bf 225} (2010) 2828--2839.

\bibitem{SS2008} S. Saraf and M. Sudan, Improved lower bound on the size of Kakeya sets over finite fields, {\it Anal. PDE} {\bf 1} (2008) 375--379.

\bibitem{Slavov2014a} K. Slavov, An algebraic geometry version of the Kakeya problem, arXiv:1410.3701.

\bibitem{Slavov2014b} K. Slavov, Variants of the Kakeya problem over an algebraically closed field, {\it Arch. Math.}, {\bf 103} (2014) 267--277.

\bibitem{Wolff1999} T. Wolff, An improved bound for Kakeya type maximal functions, {\it Rev. Mat. Iberoamericana}, {\bf 11} (1999) 651--674.
 


\end{thebibliography}
\end{document}